\documentclass[11pt, letterpaper, final]{amsart}
\usepackage{amsfonts,amsmath, amsthm, amssymb, latexsym, epsfig}
\usepackage[english]{babel}
\usepackage[utf8]{inputenc}
\usepackage[all]{xy}
\usepackage{xspace}
\usepackage{comment}
\usepackage{setspace}
\usepackage{enumerate}
\usepackage[mathscr]{eucal}
\usepackage[notcite,notref]{showkeys}

	\advance \topmargin by -\headheight
	\advance \topmargin by -\headsep

	\evensidemargin \oddsidemargin
	\newcommand{\multialg}[1]{\mathcal{M}(#1)\xspace}
	\newcommand{\corona}[1]{\mathcal{Q}(#1)\xspace}

	\newcommand{\nuc}{\mathrm{nuc}}

	\theoremstyle{plain}
	\newtheorem{thm}{Theorem}[section]
	\newtheorem{lemma}[thm]{Lemma}
	\newtheorem{theorem}[thm]{Theorem}
	\newtheorem{proposition}[thm]{Proposition}
	\newtheorem{corollary}[thm]{Corollary}

	\theoremstyle{definition}
	\newtheorem{definition}[thm]{Definition}
	
	\newtheorem{remark}[thm]{Remark}
	
	\newtheorem{example}[thm]{Example}

	\numberwithin{equation}{section}
	\numberwithin{figure}{section}

\begin{document}
	\title{A note on non-unital absorbing extensions}
	\author{James Gabe}
        \address{Department of Mathematical Sciences \\
        University of Copenhagen\\
        Universitetsparken~5 \\
        DK-2100 Copenhagen, Denmark}
        \email{jamiegabe123@hotmail.com}
        
        \thanks{This work was supported by the Danish National Research Foundation through the Centre for Symmetry and Deformation (DNRF92).}

\begin{abstract}
Elliott and Kucerovsky stated that a non-unital extension of separable $C^\ast$-algebras with a stable ideal, is nuclearly absorbing if and only if the extension is purely large. However, their proof was flawed.
We give a counter example to their theorem as stated, but establish an equivalent formulation of nuclear absorption under a very mild additional assumption to being purely large. 
In particular, if the quotient algebra is non-unital, then we show that the original theorem applies. 
We also examine how this effects results in classification theory.
\end{abstract}

\maketitle

\section{Introduction and a counter example}

A (unital) extension of $C^\ast$-algebras $0 \to \mathfrak B \to \mathfrak E \to \mathfrak A \to 0$ is called \emph{(unitally) weakly nuclear} if there is a (unital) completely positive splitting $\sigma \colon \mathfrak A \to \mathfrak E$ which is weakly nuclear,
i.e.~for every $b\in \mathfrak B$ the map $b\sigma(-)b^\ast \colon \mathfrak A \to \mathfrak B$ is nuclear. Such an extension is called trivial if we may take the weakly nuclear splitting to be a $\ast$-homomorphism. 
An extension is called (unitally) nuclearly absorbing if it absorbs every trivial, (unitally) weakly nuclear extension, i.e.~the Cuntz sum of our given extension $\mathfrak e$ with any trivial, (unitally) weakly nuclear extension is strongly unitarily equivalent
to $\mathfrak e$. A remarkable result of Elliott and Kucerovsky \cite{ElliottKucerovsky-extensions} shows that a unital, separable extension with a stable ideal is unitally nuclearly absorbing if and only if the extension is \emph{purely large}.
Recall, that an extension 
$0 \to \mathfrak B \to \mathfrak E \to \mathfrak A \to 0$ of $C^\ast$-algebras with $\mathfrak B$ stable is called purely large if for any $x\in \mathfrak E \setminus \mathfrak B$, 
the hereditary $C^\ast$-subalgebra $\overline{x\mathfrak B x^\ast}$ of $\mathfrak B$ contains a stable, $\sigma$-unital $C^\ast$-subalgebra $\mathfrak D$ which is full in $\mathfrak B$. Note that we have added the requirement that $\mathfrak D$ be $\sigma$-unital,
since this was implicitly used in \cite[Lemma 7]{ElliottKucerovsky-extensions} and since this is automatic in the separable case, which is our main concern.

In their paper, Elliott and Kucerovsky use the unital version above to obtain a non-unital version of this result, i.e.~that a non-unital extension is nuclearly absorbing if and only if it is purely large. Unfortunately this is not true.
We will provide a counter example below.

A stable $C^\ast$-algebra is said to have the \emph{corona factorisation property} if all norm-full multiplier projections are Murray--von Neumann equivalent, or equivalently, all norm-full multiplier projections are properly infinite.
As is shown in \cite{KucerovskyNg-corona}, any full extension by a $\sigma$-unital, stable $C^\ast$-algebra with the corona factorisation property, is purely large in the sense of \cite{ElliottKucerovsky-extensions}. Here full means that the Busby map is full,
i.e.~that it maps non-zero elements to full elements in the corona algebra.

$C^\ast$-algebras which do not have the corona factorisation property have rather exotic properties, see e.g.~\cite{KucerovskyNg-Sregularity}. 
It follows by \cite[Corollary 1]{Robert-nucdim} that any $\sigma$-unital, stable $C^\ast$-algebra with finite nuclear dimension, or, more generally, nuclear dimension less than $\omega$,
has the corona factorisation property. Thus for classification purposes, the corona factorisation property is not really any restriction.

After receiving an early version of this note, Efren Ruiz constructed a counter example to \cite[Theorem 4.9]{EilersRestorffRuiz-K-theoryfullext}. 
In fact, by using results from this note, Ruiz has constructed two graphs such that the induced $C^\ast$-algebras have exactly one non-trivial ideal, have isomorphic six-term exact sequences in $K$-theory with order and scale, 
but for which the $C^\ast$-algebras are non-isomorphic. 
This implies that we do not have a complete classification of graph $C^\ast$-algebras with exactly one non-trivial ideal using the above $K$-theoretic invariant, as opposed to what was previously believed.
The counter example is provided in Section 4.
Fortunately, all recent classification results of \emph{stable} graph $C^\ast$-algebras are unaffected by the issues addressed in this note, and hence stand as given.

As for general notation in this note we let $\pi$ denote the quotient map from the multiplier algebra of some $C^\ast$-algebra to its corona algebra, and we consider an essential extension algebra as a $C^\ast$-subalgebra of the multiplier algebra of the ideal. 

A counter example of \cite[Corollary 16]{ElliottKucerovsky-extensions} could be as follows.

\begin{example}\label{e:counterexm}
 Let $\mathfrak A = \mathbb C$, $\mathfrak B = \mathbb K \oplus \mathbb K$, and consider the trivial extension $\mathfrak E$ with splitting $\sigma (1) = P \oplus 1 \in \multialg{\mathbb K} \oplus \multialg{\mathbb K} \cong \multialg{\mathfrak B}$,
 where $P$ is a full projection in $\multialg{\mathbb K}$ such that $1-P$ is also full.
 The extension $\mathfrak E$ is clearly full, and since $\mathfrak B$ has the corona factorisation property, this implies that $\mathfrak E$ is a non-unital, purely large extension. 
 However, it does not absorb the zero extension, i.e.~the extension with the zero Busby map. 
 This is easily seen by projecting to the second coordinate in the corona algebra $\pi_2 \colon \corona{\mathfrak B} \cong \corona{\mathbb K}\oplus \corona{\mathbb K} \to \corona{\mathbb K}$, since $\pi_2(\tau(1)) = 1$ and $\pi_2((\tau \oplus 0)(1))$ is a non-trivial projection, where $\tau$ denotes the Busby map.
\end{example}

The flaw in the original proof is the claim that a non-unital extension $\mathfrak E$ is purely large if and only if its unitisation $\mathfrak E^\dagger$ is purely large. The sufficiency is trivial but the necessity is incorrect.

\begin{lemma}
 There exists a non-unital purely large extension such that the unitisation is not purely large.
\end{lemma}
\begin{proof}
 Let $0 \to \mathfrak B \to \mathfrak E \to \mathfrak A \to 0$ denote the extension of Example \ref{e:counterexm}. The unitisation $\mathfrak E^\dagger$ has Busby map 
 $ \tau^\dagger \colon \mathbb C \oplus \mathbb C \to \corona{\mathbb K}\oplus \corona{\mathbb K}$ given by $ \tau^\dagger(1\oplus 0) = \pi(P) \oplus 1$ and $\tau^\dagger(0 \oplus 1) = \pi(1-P) \oplus 0$. Since $\pi(1-P) \oplus 0$ is not full in
 $\corona{\mathbb K} \oplus \corona{\mathbb K}$, $\tau^\dagger$ is not a full homomorphism and thus the extension can not be purely large.
\end{proof}

\section{Fixing the theorem}

We will start by showing that the original theorem still holds, if we assume that the quotient is non-unital.

\begin{theorem}\label{t:nonunitalquot}
 Let $0 \to \mathfrak B \to \mathfrak E \to \mathfrak A \to 0$ be an extension of separable $C^\ast$-algebras with $\mathfrak B$ stable. Suppose that $\mathfrak A$ is non-unital. Then the extension is nuclearly absorbing if and only if it is purely
 large.
\end{theorem}
\begin{proof}
 As in \cite[Section 16]{ElliottKucerovsky-extensions} the extension is nuclearly absorbing if and only if the unitised extension is unitally nuclearly absorbing which in turn is equivalent to the unitised extension being purely large. Thus it suffices to show that this is equivalent
 to the non-unitised extension being purely large. We use the same proof as in the original paper. Clearly the extension is purely large if the unitisation is purely large. Assume that the non-unital extension is purely large. Note, in particular, that the Busby
 map $\tau$ is injective. It suffices to show, that
 $\overline{(1-x)\mathfrak B (1-x)^\ast}$ contains a stable $C^\ast$-subalgebra which is full in $\mathfrak B$ for any $x \in \mathfrak E$. Suppose that $(1-x) \mathfrak E \subset \mathfrak B$. Then $\pi(x)$ is a unit for $\pi(\mathfrak E) = \tau(\mathfrak A) \subset \corona{\mathfrak B}$.
 However, this contradicts that $\mathfrak A$ is non-unital, since the Busby map $\tau$ is injective. Hence we may find $x' \in \mathfrak E$ such that $(1-x)x' \notin \mathfrak B$. Since
 \[
  \overline{(1-x)x' \mathfrak B ((1-x)x')^\ast} \subset \overline{(1-x) \mathfrak B (1-x)^\ast}
 \]
 and since the non-unital extension is purely large, the former of these contains a stable $C^\ast$-subalgebra which is full in $\mathfrak B$.
\end{proof}

To prove a stronger result, where the assumption that the quotient being unital is removed, we will use the following lemma.

\begin{lemma}\label{l:quotcomplex}
Let $\mathfrak B$ be a stable, separable $C^\ast$-algebra, and let $P\in \multialg{\mathfrak B}$ be a norm-full, properly infinite projection. 
Then the trivial extension of $\mathbb C$ by $\mathfrak B$ with splitting $\sigma$ given by $\sigma(1) = P$, is purely large.
\end{lemma}
\begin{proof}
 If $P=1$ then the extension is the canonical unitisation extension which is clearly self-absorbing. It follows from \cite{ElliottKucerovsky-extensions} that it is purely large.

It is well known, since $\mathfrak B$ is stable, that $P$ is full and properly infinite exactly when it is Murray--von Neumann equivalent to $1$. 
Let $v$ be an isometry such that $vv^\ast = P$ and let $t_1,t_2 \in \multialg{\mathfrak B}$ be such that $t_1t_1^\ast + t_2 t_2^\ast =P = t_1^\ast t_1 = t_2^\ast t_2$. 
Then $s_1:= t_1v$ and $s_2:= t_2 + (1-P)$ are the canonical generators of a unital copy of $\mathcal O_2$ in $\multialg{\mathfrak B}$, for which $P = s_1s_1^\ast + s_2 Ps_2^\ast$. Hence
\[
\pi (\sigma(1)) = \pi(s_1) 1 \pi(s_1)^\ast + \pi(s_2) \pi(P) \pi(s_2)^\ast,
\]
which implies that our extension is the Cuntz sum of the unitisation extension and itself. It follows from \cite[Lemma 13]{ElliottKucerovsky-extensions} that our extension is purely large.
\end{proof}

Now for the stronger case where we allow the quotient to be unital.

\begin{theorem}\label{t:unitalquot}
  Let $0 \to \mathfrak B \to \mathfrak E \to \mathfrak A \to 0$ be an extension of separable $C^\ast$-algebras with $\mathfrak B$ stable. 
  The extension is nuclearly absorbing if and only if it is purely large and there is a norm-full, properly infinite projection $P\in \multialg{\mathfrak B}$ such that $P\mathfrak E \subset \mathfrak B$.
\end{theorem}
\begin{proof}
 Assume that the extension is nuclearly absorbing. Then it absorbs the zero extension so we may assume that the Busby map is of the form $\tau\oplus 0$, where $\oplus$ denotes a Cuntz sum. 
 Let $P = 0\oplus 1$. Then $P \mathfrak E \subset \mathfrak B$ since $\pi(P)$ annihilates the image of the Busby map. Moreover, the extension absorbs some purely large extension and is thus itself purely large by \cite[Lemma 13]{ElliottKucerovsky-extensions}.
 
 Now suppose that the extension is purely large and that $P$ is a full, properly infinite projection such that $P\mathfrak E \subset \mathfrak B$. As in the proof of Theorem \ref{t:nonunitalquot} it suffices to show
 that the unitised extension is purely large. It is enough to show that $\overline{(1-x
) \mathfrak B (1-x)^\ast}$ contains a stable $C^\ast$-subalgebra which is full in $\mathfrak B$, for
 any $x\in \mathfrak E$. Observe that
 \[
  \overline{(1-x)P\mathfrak B P(1-x)^\ast} \subset \overline{(1-x) \mathfrak B (1-x)^\ast}.
 \]
Since $(1-x)P = P - xP$ and $xP \in \mathfrak B$, it suffices to show that the extension $0 \to \mathfrak B \to \mathfrak B + \mathbb C P \to \mathbb C \to 0$ is purely large. This follows from Lemma \ref{l:quotcomplex}.
\end{proof}

Note that an extension must be non-unital in order to satisfy the equivalent conditions in the above theorem. We immediately get the following corollary. 

\begin{corollary}\label{c:abszero}
  Let $0 \to \mathfrak B \to \mathfrak E \to \mathfrak A \to 0$ be an extension of separable $C^\ast$-algebras with $\mathfrak B$ stable.
  Then the extension is nuclearly absorbing if and only if it is purely large and absorbs the zero extension.
\end{corollary}

When we assume that the ideal has the corona factorisation property, then we get a perhaps more hands on way for checking if a full extension is nuclearly absorbing. To exhibit this we introduce the following definition.

\begin{definition}\label{d:unitfull}
Let $0 \to \mathfrak B \to \mathfrak E \to \mathfrak A \to 0$ be an extension of $C^\ast$-algebras. We say that the extension is \emph{unitisably full} if the unitised extension $0 \to \mathfrak B \to \mathfrak E^\dagger \to \mathfrak A^\dagger \to 0$ is full.
\end{definition}

It is clear that if an extension is unitisably full, then it is full and non-unital. If the quotient algebra $\mathfrak A$ is unital, then the extension is unitisably full if and only if the extension is full and $1_{\corona{\mathfrak B}}-\tau(1_{\mathfrak A})$ is full, where $\tau$ denotes the Busby map. Note that this case is our main concern due to Theorem \ref{t:nonunitalquot}.

\begin{theorem}\label{t:coronaunitfull}
 Let $0 \to \mathfrak B \to \mathfrak E \to \mathfrak A \to 0$ be an extension of separable $C^\ast$-algebras, such that $\mathfrak B$ is stable and has the corona factorisation property. Then the extension is nuclearly
 absorbing if and only if the extension is unitisably full.
\end{theorem}
\begin{proof}
As in the proof of Theorem \ref{t:nonunitalquot} the extension is nuclearly absorbing if and only if the unitised extension is purely large. Since $\mathfrak B$ has the corona factorisation property, this is the case if and only if the extension is unitisably full.
\end{proof}

We will end this section by showing that in the absence of the corona factorisation property, there are purely large, unitisably full extensions which are not nuclearly absorbing. We will need a converse of Lemma \ref{l:quotcomplex}.

\begin{proposition}\label{p:extofC}
Let $\mathfrak B$ be a stable, separable $C^\ast$-algebra, and let $P\in \multialg{\mathfrak B}$ be a norm-full projection. 
Then the trivial extension of $\mathbb C$ by $\mathfrak B$ with splitting $\sigma$ given by $\sigma(1) = P$, is purely large if and only if $P$ is properly infinite.
\end{proposition}
\begin{proof}
 One direction is Lemma \ref{l:quotcomplex}. Suppose that the extension is purely large. It suffices to show that the Cuntz sum $P\oplus 0$ is properly infinite. The extension with splitting $\sigma'(1) = P\oplus 0$ is purely large and absorbs the zero extension,
 and thus it is absorbing by Corollary \ref{c:abszero}. Since the extension with splitting $\sigma_0 (1) = 1\oplus 0$ is also absorbing, there is a unitary $U\in \multialg{\mathfrak B}$ such that $U^\ast (P \oplus 0) U - 1\oplus 0 \in \mathfrak B$. 
 Pick an isometry $V \in \multialg{\mathfrak B}$ such that $V^\ast (1\oplus 0) V = 1$. Then $V^\ast( U^\ast (P \oplus 0) U - 1\oplus 0)V = (UV)^\ast (P\oplus 0) UV - 1 \in \mathfrak B$. Since $\mathfrak B$ is stable, we may find an isometry $W$ such that
 \[ 
  \| (UVW)^\ast (P\oplus 0) UVW - 1\| = \| W^\ast( (UV)^\ast (P\oplus 0) UV - 1 ) W\| < 1.
 \]
 This implies that $P\oplus 0$, and thus also $P$, is properly infinite.
\end{proof}

We can now extend our class of counter examples to include purely large, unitisably full extensions $0 \to \mathfrak B \to \mathfrak E \to \mathfrak A \to 0$, which are not nuclearly absorbing. 
In fact, such an extension can be made for any $\mathfrak B$ without the corona factorisation property.

\begin{proposition}\label{p:noncoronanonabs}
 Let $\mathfrak B$ be a stable, separable $C^\ast$-algebra which does not have the corona factorisation property. Then there is a purely large, unitisably full extension of $\mathbb C$ by $\mathfrak B$ which is not nuclearly absorbing.
\end{proposition}
\begin{proof}
 Let $Q$ be a full multiplier projection which is not properly infinite, but where $P:= 1-Q$ is properly infinite and full. 
 Such a projection can be obtained by taking any full multiplier projection $Q'$ which is not properly infinite, and letting $Q = Q' \oplus 0$ be a Cuntz sum. In fact, $P=1-Q$ will be properly infinite since it majorises the properly infinite, full projection $0\oplus 1$. 
 Consider the trivial extension $\mathfrak E$ of $\mathbb C$ by $\mathfrak B$ with splitting $\sigma(1) = P$. The unitised extension has a splitting $\sigma_1 \colon \mathbb C \oplus \mathbb C \to \multialg{\mathfrak B}$ given by $\sigma_1(1 \oplus 0) = P$ and $\sigma_1(0\oplus 1) = Q$. Since both $P$ and $Q$ are full and orthogonal, the unitised extension is full. 
 
 By Proposition \ref{p:extofC} the extension is purely large. As seen in \cite{ElliottKucerovsky-extensions} such an extension is nuclearly absorbing exactly when its unitisation is purely large. If the unitisation was purely large,
 then $\overline{(Q-b)\mathfrak B(Q-b)^\ast}$ would contain a stable $C^\ast$-subalgebra full in $\mathfrak B$, for every $b\in \mathfrak B$. 
 However, this would imply that the extension of $\mathbb C$ by $\mathfrak B$ with splitting $\sigma_0(1) = Q$ is purely large, which it is not by Proposition \ref{p:extofC}. Hence the extension is not nuclearly absorbing.
\end{proof}

\section{How this affects classification results}

In the classification of non-simple $C^\ast$-algebras, a popular result has been a result of Kucerovsky and Ng, which says that under the mild condition of the corona factorisation property on a stable, separable $C^\ast$-algebra $\mathfrak B$, 
$KK^1(\mathfrak A,\mathfrak B)$ is the group of unitary equivalence classes of full extensions $\mathfrak E$ of $\mathfrak A$ by $\mathfrak B$ for any nuclear separable $C^\ast$-algebra $\mathfrak A$. This is unfortunately not the case.
The theorem only remains true if one adds the condition that the extensions are unitisably full as in Definition \ref{d:unitfull}. See Theorem \ref{t:KK1class} below.

A counter example of the original result could be as follows.

\begin{example}
Let $0 \to \mathfrak B \to \mathfrak E \to \mathfrak A \to 0$ be the extension from Example \ref{e:counterexm} with Busby map $\tau$. Then $\mathfrak B$ has the corona factorisation property and the extension is full. 
As seen in Example \ref{e:counterexm}, $\tau$ and $\tau \oplus 0$ are both non-unital and are \emph{not} unitarily equivalent. However, they define the same element in $KK^1(\mathfrak A, \mathfrak B)$.
\end{example}

The closest we get to fixing the theorem would be the following.

\begin{theorem}\label{t:KK1class}
 Let $\mathfrak B$ be a separable, stable $C^\ast$-algebra. Then the following are equivalent.
 \begin{itemize}
  \item[$(i)$] $\mathfrak B$ has the corona factorisation property,
  \item[$(ii)$] for any separable $C^\ast$-algebra $\mathfrak A$, $KK^1_\nuc(\mathfrak A, \mathfrak B)$ is the group of strong unitary equivalence classes of all full, weakly nuclear extensions of $\mathfrak A$ by $\mathfrak B$ which absorb the zero extension,
  \item[$(iii)$] for any separable $C^\ast$-algebra $\mathfrak A$, $KK^1_\nuc(\mathfrak A, \mathfrak B)$ is the group of strong unitary equivalence classes of all full, weakly nuclear extensions $\mathfrak E$ of $\mathfrak A$ by $\mathfrak B$, 
  for which there is a norm-full projection $P\in \multialg{\mathfrak B}$ such that $P \mathfrak E \subset \mathfrak B$.
  \item[$(iv)$] for any separable $C^\ast$-algebra $\mathfrak A$, $KK^1_\nuc(\mathfrak A, \mathfrak B)$ is the group of strong unitary equivalence classes of all unitisably full, weakly nuclear extensions $\mathfrak E$ of $\mathfrak A$ by $\mathfrak B$.
 \end{itemize}
\end{theorem}
\begin{proof}
 It is well-known that $KK^1_\nuc(\mathfrak A, \mathfrak B)$ is (isomorphic to) the group of strong unitary equivalence classes of weakly nuclear extensions of $\mathfrak A$ by $\mathfrak B$ which are nuclearly absorbing. Thus $(i)\Rightarrow (iv)$ by Theorem \ref{t:coronaunitfull} and
 $(iv) \Rightarrow (i)$ follows from Proposition \ref{p:noncoronanonabs}.
 
 If $\mathfrak B$ has the corona factorisation property, then any full extension by $\mathfrak B$ is purely large. Thus $(i) \Rightarrow (iii)$ follows from Theorem \ref{t:unitalquot}.
 
 $(iii)$ is clearly equivalent to the condition, that for any $C^\ast$-algebra $\mathfrak A$, any full, weakly nuclear extension $\mathfrak E$ of $\mathfrak A$ by $\mathfrak B$, for which there is a full projection $P\in \multialg{\mathfrak B}$ such that 
 $P\mathfrak E \subset \mathfrak B$, is nuclearly absorbing. If the extension $\mathfrak E$ has Busby map $\tau \oplus 0$, then $(0\oplus 1)\mathfrak E \subset \mathfrak B$, and thus $(iii) \Rightarrow (ii)$.
 
 It remains to show $(ii) \Rightarrow (i)$. Let $P\in \multialg{\mathfrak B}$ be a full projection, and $P \oplus 0$ be the Cuntz sum. Note that $Q \sim Q \oplus 0$ for any projection $Q$.
 By $(ii)$ the extension with the Busby map $\tau \colon \mathbb C \to \corona{\mathfrak B}$ given by $\tau(1) = \pi(P \oplus 0)$ is nuclearly absorbing. In particular, it absorbs the unitisation extension of $\mathfrak B$. 
 Consider the lift $\rho(1) = P \oplus 0$ of $\tau$ and the canonical lift of the unitisation extension of $\mathfrak B$. 
 We may find a unitary $u\in \multialg{\mathfrak B}$ such that $u^\ast (P \oplus 0 \oplus 0)u - 0 \oplus 0 \oplus 1 \in \mathfrak B$. If $v$ is an isometry such that $vv^\ast = 0 \oplus 0 \oplus 1$
 then $(uv)^\ast (P \oplus 0 \oplus 0) uv - 1 \in \mathfrak B$. Thus we may pick an isometry $w$ such that 
 \[
  \| w^\ast((uv)^\ast (P\oplus 0 \oplus 0) uv - 1)w\| = \| s^\ast (P \oplus 0 \oplus 0)s - 1\|  < 1,
 \]
 where $s$ is the isometry $uvw$. 
 Hence $P$ is Murray--von Neumann equivalent to $s^\ast(P \oplus 0 \oplus 0) s$, which is equivalent to $1$. 
\end{proof}

\begin{remark}
 It clearly follows from the proof above, that we could restrict our attention only to nuclear $C^\ast$-algebras $\mathfrak A$ if desired. In this case we can remove the weakly nuclear condition, since all extensions of a separable, nuclear $C^\ast$-algebra 
 are weakly nuclear by the lifting theorem of Choi and Effros \cite{ChoiEffros-lifting}, and also we would have $KK^1_\nuc(\mathfrak A, \mathfrak B) = KK^1(\mathfrak A , \mathfrak B)$.
\end{remark}

We still get some nice results for classification. This follows from the above theorem and Theorem \ref{t:nonunitalquot}.

\begin{corollary}\label{c:nonunitalclass}
 Let $\mathfrak B$ be a separable, stable $C^\ast$-algebra with the corona factorisation property and let $\mathfrak A$ be a non-unital, separable $C^\ast$-algebra. Then $KK^1_\nuc(\mathfrak A, \mathfrak B)$ is the group of strong unitary equivalence classes
 of all full, weakly nuclear extensions of $\mathfrak A$ by $\mathfrak B$.
\end{corollary}

\begin{corollary}\label{c:stableclass}
 Let $\mathfrak B$ be a separable, stable $C^\ast$-algebra with the corona factorisation property and let $\mathfrak A$ be a separable $C^\ast$-algebra. Let $\mathfrak E_i$ be full, weakly nuclear extensions of $\mathfrak A$ by $\mathfrak B$, 
 with Busby maps $\tau_i$, for $i=1,2$. If $[\tau_1] = [\tau_2] \in KK^1_\nuc(\mathfrak A, \mathfrak B)$, then $\mathfrak E_1 \otimes \mathbb K \cong \mathfrak E_2 \otimes \mathbb K$.
\end{corollary}
\begin{proof}
 Given a Busby map $\tau \colon \mathfrak A \to \corona{\mathfrak B}$, let $\tau^s$ be the composition 
 \[
 \mathfrak A \otimes \mathbb K \xrightarrow{\tau \otimes id} \corona{\mathfrak B} \otimes \mathbb K \hookrightarrow \corona{\mathfrak B \otimes \mathbb K}.
 \]
 It is well-known that the map $KK^1_\nuc(\mathfrak A, \mathfrak B) \to KK^1_\nuc(\mathfrak A\otimes \mathbb K, \mathfrak B \otimes \mathbb K)$ given by $[\tau] \mapsto [\tau^s]$, is an isomorphism 
 (the proof is identical to the similar result in classical $KK$-theory). Thus $\tau_1^s$ and $\tau_2^s$ are strongly unitarily equivalent 
 by Corollary \ref{c:nonunitalclass}, and since their corresponding extension algebras are $\mathfrak E_1 \otimes \mathbb K$ and $\mathfrak E_2 \otimes \mathbb K$ respectively, it follows that $\mathfrak E_1 \otimes \mathbb K \cong \mathfrak E_2 \otimes \mathbb K$.
\end{proof}

\begin{remark}
Every result in this note holds with the ideal $\mathfrak B$ being $\sigma$-unital instead of separable. The quotient $\mathfrak A$ should still be separable. 
This is a special case of a much more general result, proved by the author in collaboration with Efren Ruiz in \cite{GabeRuiz-absrep}.
\end{remark}

\section{The counter example of Ruiz}

\begin{definition}
 Let $E=(E^0,E^1,r,s)$ be a (countable, directed) graph. The \emph{graph $C^\ast$-algebra} $C^\ast(E)$ is the universal $C^\ast$-algebra generated by mutually orthogonal projections $p_v$ for $v\in E^0$, and isometries $s_e$ for $e\in E^1$,
 which satisfy the relations
 \begin{itemize}
  \item $s_e^\ast s_f = \delta_{ef} p_{r(e)}$ for all $e,f\in E^1$,
  \item $s_e s_e^\ast \leq p_{s(e)}$ for all $e\in E^1$,
  \item $p_v = \sum_{e\in s^{-1}(\{ v\})} s_e s_e^\ast$ for all $v\in E^0$ satisfying $0 < | s^{-1}(\{ v\} ) | < \infty$.
 \end{itemize}
\end{definition}

\begin{example}[Counter example to Theorem 4.9 of \cite{EilersRestorffRuiz-K-theoryfullext}]
 \cite[Theorem 4.9]{EilersRestorffRuiz-K-theoryfullext} states that if $C^\ast(E)$ and $C^\ast(F)$ are non-unital and both have exactly one non-trivial, two-sided, closed ideal, and the induced six-term exact sequence in $K$-theory are isomorphic, such that
 the isomorphisms on all $K_0$-groups preserve order and scale, then $C^\ast(E) \cong C^\ast(F)$. We will provide a counter example to this result.

 Let $E$ and $F$ be the respective graphs
 
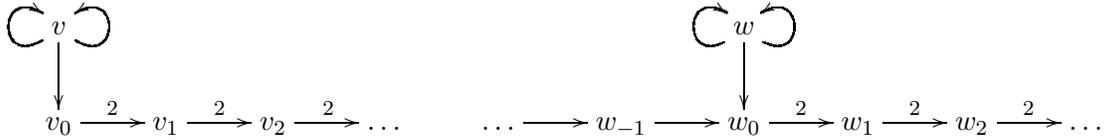
\[
  \xymatrix{
  v \ar@(rd,ru) \ar@(ld,lu) \ar[d] &&&&&& w \ar@(rd,ru) \ar@(ld,lu) \ar[d]\\
  v_0 \ar[r]^2 & v_1 \ar[r]^2 & v_2 \ar[r]^2 & \dots& \dots \ar[r] &  w_{-1} \ar[r] &  w_{0} \ar[r]^2 & w_1 \ar[r]^2 & w_2 \ar[r]^2 & \dots
  }
 \]
 Both $C^\ast(E)$ and $C^\ast(F)$ are non-unital, full extensions of the Cuntz algebra $\mathcal O_2$ by the stabilisation of the CAR algebra $M_{2^\infty} \otimes \mathbb K$.
 The six-term exact sequences of the induced extensions, where we write the $K_0$-groups with order and scale as $(K_0(\mathfrak A), K_0(\mathfrak A)^+, \Sigma K_0(\mathfrak A))$, are both isomorphic to
 \[
  \xymatrix{
  (\mathbb Z[\tfrac{1}{2}], \mathbb Z[\tfrac{1}{2}]_+,\mathbb Z[\tfrac{1}{2}]_+)  \ar[r]^{\; \; (id,\iota,\iota)} & (\mathbb Z[\tfrac{1}{2}],\mathbb Z[\tfrac{1}{2}],\mathbb Z[\tfrac{1}{2}]) \ar[r] & (0,0,0) \ar[d] \\
  \ar[u] 0 & 0 \ar[l] & 0. \ar[l]
  }
 \]
 where $\mathbb Z[\tfrac{1}{2}]_+ = \mathbb Z[\tfrac{1}{2}] \cap [0,\infty)$ and $\iota \colon \mathbb Z[\tfrac{1}{2}]_+ \hookrightarrow \mathbb Z[\tfrac{1}{2}]$ is the canonical inclusion.
 To compute the order and scale of $K_0(C^\ast(E))$ and $K_0(C^\ast(F))$ we simply used that both $C^\ast(E)$ and $C^\ast(F)$ contain a full, properly infinite projection, $p_v$ and $p_w$ respectively, and applied \cite[Proposition 4.1.4]{Rordam-book-class}.
 Thus if \cite[Theorem 4.9]{EilersRestorffRuiz-K-theoryfullext} were true, it should follow that $C^\ast(E) \cong C^\ast(F)$. We will show that this is not the case, by showing that one extension with $C^\ast(F)$ is nuclearly absorbing, but that the
 extension with $C^\ast(E)$ is not.
 
 \textbf{The extension with $C^\ast(F)$ is nuclearly absorbing:}
 Recall, that $F^\ast$ denotes the set of paths in $F$, and that if $\alpha = e_1 \dots e_n\in F^\ast$ then $s_\alpha := s_{e_1} \dots s_{e_n}$, and that $r(\alpha) = r(e_n)$ and $s(\alpha) = s(e_1)$.
 The ideal $\mathfrak I_F$ in $C^\ast(F)$ isomorphic to $M_{2^\infty} \otimes \mathbb K$ is given by
 \[
  \mathfrak I_F = \overline{\mathrm{span}} \{ s_\alpha s_\beta^\ast : \alpha, \beta\in F^\ast, r(\alpha) = r(\beta) = w_n \text{ for some }n \in \mathbb Z \}.
 \]
 Let $P = \sum_{n=1}^\infty w_{-n}$ which is easily seen to converge strictly in the multiplier algebra of $\mathfrak I_F$.
 We clearly have that $P C^\ast(F) \subset \mathfrak I_F$. Thus, if $P$ is a full, properly infinte projection in $\multialg{\mathfrak I_F}$, then it follows from Theorem \ref{t:unitalquot} that the extension with $C^\ast(F)$ is nuclearly absorbing.
 Since $\mathfrak I_F$ has the corona factorisation property it suffices to show that $P$ is full.
 
 Note that $M_{2^\infty} \cong p_{w_0} \mathfrak I_F p_{w_0}$. Let $\rho$ denote the unique tracial state on $p_{w_0} \mathfrak I_F p_{w_0}$, and $\rho_\infty$ denote the induced trace function on $\multialg{\mathfrak I_F}_+$.
 It follows from \cite[Theorem 4.4]{Rordam-multialg} that $P$ is full if and only if $\rho_\infty(P) = \infty$. 
 Since $p_{w_{-n}}$ for $n>0$ is Murray--von Neumann equivalent to $p_{v_0}$, it follows that $\rho(p_{w_{-n}}) = \rho(p_{v_0}) = 1$ and thus $\rho_\infty (P) = \sum_{n=1}^\infty \rho(p_{w_{-n}}) = \infty$.
 Thus the extension $0 \to \mathfrak I_F \to C^\ast(F) \to C^\ast(F)/\mathfrak I_F \to 0$ is nuclearly absorbing.
 
 \textbf{The extension with $C^\ast(E)$ is not nuclearly absorbing:} 
 The ideal $\mathfrak I_E$ in $C^\ast(E)$ isomorphic to $M_{2^\infty} \otimes \mathbb K$ is given by
 \[
  \mathfrak I_E = \overline{\mathrm{span}} \{ s_\alpha s_\beta^\ast : \alpha, \beta\in E^\ast, r(\alpha) = r(\beta) = v_n \text{ for some }n \geq 0 \}.
 \]
 To show that the extension $\mathfrak e \colon 0 \to \mathfrak I_E \to C^\ast(E) \to \mathcal O_2 \to 0$ is not nuclearly absorbing, it suffices to show that the unitised extension
 $\mathfrak e^\dagger \colon 0 \to \mathfrak I_E \to C^\ast(E)^\dagger \to \mathcal O_2 \oplus \mathbb C \to 0$ is \emph{not} full.
 Let $\sigma \colon C^\ast(E) \to \multialg{\mathfrak I_E}$ be the canonical $\ast$-homomorphism. Then $C^\ast(E)^\dagger \cong \sigma(C^\ast(E)) + 1_{\multialg{\mathfrak I_E}}$.
 Note that $1-\sigma(p_{v})$ is a lift of $(0,1) \in \mathcal O_2 \oplus \mathbb C$ (under the obvious identifications),
 so if $\mathfrak e^\dagger$ is full, we should have that $1-\sigma(p_{v}) + \mathfrak I_E$ is full in $\corona{\mathfrak I_E}$.
 Since $\mathfrak I_E$ is stable, fullness of $1-\sigma(p_{v}) + \mathfrak I_E$ is equivalent to fullness of $1- \sigma(p_v)$ in $\multialg{\mathfrak I_E}$.
 
 The corner in $\mathfrak I_E$ generated by $1-\sigma(p_v)$ is easily seen to be
 \[
  \overline{\mathrm{span}} \{ s_\alpha s_\beta^\ast : s(\alpha) \neq v \neq s(\beta)\},
 \]
 which has an approximate unit $(\sum_{n=0}^k p_{v_n})_{k=1}^\infty$. Thus $1 - \sigma(p_v) = \sum_{n=0}^\infty p_{v_n}$.
 As above, $M_{2^\infty} \cong p_{v_0} \mathfrak I_E p_{v_0}$, so let $\rho$ be the unique tracial state and $\rho_\infty$ be the induced trace function on $\multialg{\mathfrak I_E}_+$.
 We have that $\rho(p_{v_n}) = 2^{-n}$ so $\rho_\infty (1-\sigma(p_v)) = \sum_{n=0}^\infty 2^{-n} < \infty$. It follows that $1-\sigma(p_v)$ is not full, and thus $\mathfrak e$ is not nuclearly absorbing.
 
 In particular, $C^\ast(E) \not \cong C^\ast(F)$.
\end{example}

\subsection*{Acknowledgement} I am indebted to Efren Ruiz for the above counter example, as well as for other helpful comments.


\begin{thebibliography}{ERR14}

\bibitem[CE76]{ChoiEffros-lifting}
Choi, M.~D. and Effros, E.~G.
\newblock The completely positive lifting problem for {$C\sp*$}-algebras.
\newblock {\em Ann. of Math.}, 104 (3), 585--609, 1976.

\bibitem[EK01]{ElliottKucerovsky-extensions}
Elliott, G.~A. and Kucerovsky, D.
\newblock An abstract {V}oiculescu-{B}rown-{D}ouglas-{F}illmore absorption
  theorem.
\newblock {\em Pacific J. Math.}, 198 (2), 385--409, 2001.

\bibitem[ERR14]{EilersRestorffRuiz-K-theoryfullext}
Eilers, S., Restorff, G.~and Ruiz, E.
\newblock The ordered {$K$}-theory of a full extension.
\newblock {\em Canad. J. Math.}, 66 (3), 596--625, 2014.

\bibitem[GR15]{GabeRuiz-absrep}
James Gabe and Efren Ruiz.
\newblock Absorbing representations with respect to operator convex cones.
\newblock {\em preprint, arXiv:1503.07799}, 2015.

\bibitem[KN06a]{KucerovskyNg-Sregularity}
Kucerovsky, D.~and Ng, P.~W.
\newblock {$S$}-regularity and the corona factorization property.
\newblock {\em Math. Scand.}, 99 (2), 204--216, 2006.

\bibitem[KN06b]{KucerovskyNg-corona}
Kucerovsky, D.~and Ng, P.~W.
\newblock The corona factorization property and approximate unitary
  equivalence.
\newblock {\em Houston J. Math.}, 32 (2), 531--550, 2006.

\bibitem[Rob11]{Robert-nucdim}
Robert, L.
\newblock Nuclear dimension and {$n$}-comparison.
\newblock {\em M\"unster J. Math.}, 4, 65--71, 2011.

\bibitem[R{\o}r91]{Rordam-multialg}
Rørdam, M.
\newblock Ideals in the multiplier algebra of a stable {$C^\ast$}-algebra
\newblock {\em J. Operator Theory}, 25 (2), 283--298, 1991.

\bibitem[R{\o}r02]{Rordam-book-class}
M.~R{\o}rdam.
\newblock Classification of nuclear, simple {$C^*$}-algebras.
\newblock In {\em Classification of nuclear {$C^*$}-algebras. {E}ntropy in
  operator algebras}, volume 126 of {\em Encyclopaedia Math. Sci.}, pages
  1--145. Springer, Berlin, 2002.

\end{thebibliography}
\end{document}